\documentclass[11pt,a4paper]{article}

\usepackage{theorem,enumerate}
\usepackage{amsmath,latexsym,amssymb,amsfonts}
\usepackage{eucal}
\usepackage{mathrsfs}
\usepackage{array}
\usepackage{footmisc}
\usepackage{graphicx, color}
\usepackage{caption,subcaption}

\newtheorem{theorem}{Theorem} \rm
\newtheorem{lemma}[theorem]{Lemma}

\newtheorem{claim}[theorem]{Claim}

\newtheorem{definition}[theorem]{Definition}

\numberwithin{theorem}{section}
\usepackage{color}

\newcommand{\qed}{\hfill \ensuremath{\Box}}

\begin{document}
\title{\bf The Alon-Tarsi number of planar graphs without  cycles of lengths $4$ and $l$}

\author{
	Huajing Lu\thanks{Department of Mathematics, Zhejiang Normal University, China.
		e-mail: {\tt huajinglu@zjnu.edu.cn
		}} \and
	Xuding Zhu\thanks{Department of Mathematics, Zhejiang Normal University, China.
		e-mail: {\tt  xdzhu@zjnu.edu.cn
		}} \thanks{Grant Numbers: NSFC 11571319 and 111 project of Ministry of Education of China.}
	}

\maketitle

\begin{abstract}
This paper proves that if $G$ is a planar graph without 4-cycles and $l$-cycles for some $l\in\{5, 6, 7\}$, then there exists a matching $M$ such that $AT(G-M)\leq 3$. This implies that every planar graph without 4-cycles and $l$-cycles for some $l\in\{5, 6, 7\}$ is 1-defective 3-paintable.

\end{abstract}
\par\bigskip\noindent
 {\em Keywords:} planar graph; choice number; paint number; Alon-Tarsi number.

\section{Introduction}
Assume $G$ is a graph and $d$ is a non-negative integer.
A {\em $d$-defective coloring} of $G$ is a coloring of the vertices of $G$ such that each color class induces a subgraph of maximum degree at most $d$. A 0-defective coloring of $G$ is also called  a {\em proper coloring} of $G$. In a coloring of the vertices of $G$, we say an edge $e$ is a {\em fault edge}  if the end vertices of $e$ receive the same color. A coloring of $G$ is 1-defective if and only if the set of fault edges is a matching. A {\em k-list assignment} of a graph $G$ is a mapping $L$ which assigns to each vertex $v$   a set $L(v)$ of $k$ permissible colors. Given a $k$-list assignment $L$ of $G$, a {\em $d$-defective $L$-coloring} of $G$ is a $d$-defective coloring $c$ of $G$ with $c(v)\in L(v)$ for every vertex $v$ of $G$. A graph $G$ is {\em $d$-defective $k$-choosable} if for any $k$-list assignment $L$ of $G$, there exists a $d$-defective $L$-coloring of $G$. We say $G$ is  {\em $k$-choosable}  if $G$ is $0$-defective $k$-choosable. The {\em choice number} $ch(G)$ of $G$ is the minimum $k$ for which $G$ is $k$-choosable.

Defective list coloring of planar graphs has been studied in a few papers.
Eaton and Hall \cite{EH}, and \v{S}krekovski \cite{Skr} proved independently that every planar graph is $2$-defective $3$-choosable. Cushing and Kierstead \cite{CK} proved that every planar graph is $1$-defective $4$-choosable. The above results can be reformulated as follows:

Assume $G$ is a planar graph. (1) For every $3$-list assignment $L$ of $G$, there is a subgraph $H$ of $G$ with $\Delta(H)\le 2$ and $G-E(H)$ is $L$-colorable.  (2) For every $4$-list assignment $L$ of $G$, there is a subgraph $H$ of $G$ with $\Delta(H)\le 1$ and $G-E(H)$ is $L$-colorable.

In the proofs of \cite{CK}, \cite{EH} and \cite{Skr}, the subgraph $H$ depends on the list assignment $L$. A natural question is whether there is a subgraph $H$ that does not depend on $L$. 
In other words, we ask the following questions:

{\em  
	 (1) Is it true that every planar graph $G$ has a subgraph $H$ with $\Delta(H)\le2$ such that $G-E(H)$ is $3$-choosable?}

{\em  
	(2) 
	 Is it true that every planar graph $G$ has a subgraph $H$ with $\Delta(H)\le1$, such that $G-E(H)$ is $4$-choosable?
}

It turns out that the answer to (1)  is negative and the answer to (2) is positive. Very recently, it is shown in \cite{KKZ} that  
there is a planar graph $G$ such that for any subgraph $H$ of $G$ with $\Delta(H)\le3$ (this number $3$ is not a typo), $G-E(H)$ is not $3$-choosable. On the other hand, as a consequence of the main  result 
 in \cite{GZ}, every planar graph $G$ has a matching $M$ such that $G-M$ is $4$-choosable.

The main result   in \cite{GZ} is about the Alon-Tarsi number of $G-M$.
We associate to each vertex $v$ of $G$ a variable $x_v$. The graph polynomial $P_G({\bf x})$ of $G$ is defined as $P_G({\bf x})=\prod_{u \sim v,u<v}(x_v-x_u)$, where ${\bf x}=\{x_v: v\in V(G)\}$ and $<$ is an arbitrary fixed ordering of the vertices of $G$. It is easy to see that a mapping $\phi : V\rightarrow R$ is a proper coloring of $G$ if and only if $P_G(\phi)\ne0$, where $P_G(\phi)$ means to evaluate the polynomial at $x_v=\phi(v)$ for $v\in V(G)$. Thus to find a proper coloring of $G$ is equivalent to find an assignment of ${\bf x}$ so that the polynomial evaluated at the assignment is non-zero.

For a mapping $\eta: V(G) \to \{0,1,\ldots\}$, let $c_{P_G, \eta}$ be the coefficient of the monomial $\prod_{v \in V(G)} x_v^{\eta(v)}$ in the expansion of $P_G$. It follows from the Combinatorial Nullstellensatz   that if   $c_{P, \eta}\ne 0$, and $L$ is a list assignment of $G$ for which   $|L(v)| = \eta(v)+1$, then $G$ is $L$-colorable. (Note that $P_G$ is a homogeneous polynomial, and all the monomials with nonzero coefficient are of highest degree.)   In particular, if $c_{P_G, \eta}\ne0$ and $\eta( v)<k$ for all $v\in V(G)$, then $G$ is $k$-choosable.
Jensen and Toft \cite{JT} defined the {\em Alon-Tarsi number} of $G$ as
\begin{center}
$AT(G)$=min\{$k:c_{P_G, \eta}\ne0$ for some $\eta$ with $\eta(v)<k$ for all $v\in V(G)$\}.
\end{center}
Thus for any graph $G$, $ch(G)\le AT(G)$. The following is the main result in \cite{GZ}.

\begin{theorem}
  \label{thm-GZ}
  Every planar graph $G$ has a matching $M$ such that $AT(G-M)\leq 4$.
\end{theorem}

This theorem actually implies the online version of $1$-defective $4$-choosability of planar graphs. The online version of $d$-defective $k$-choosable is called {\em $d$-defective $k$-paintable} and is defined through a two-person game. 

Given a graph $G$ and non-negative integers $k,d$, the 
{\em  $d$-defective $k$-painting game} on $G$ is played by two players: Lister and Painter. Initially, each vertex has $k$ tokens and is uncolored. In each round, Lister selects a nonempty set $M$ of uncolored vertices and takes away one token from each vertex in $M$. Painter colors a subset $X$ of $M$ such that the induced subgraph $G[X]$ has maximum degree at most $d$. If at the end of a certain round, there is an uncolored vertex with no tokens left, then Lister wins. Otherwise, at the end of some round, all vertices are colored, Painter wins. We say $G$ is {\em $d$-defective $k$-paintable} if Painter has a winning strategy in the $d$-defective $k$-painting game. The 0-defective $k$-painting game is also called the $k$-painting game, and we say $G$ is {\em $k$-paintable} if it is 0-defective $k$-paintable. The {\em paint number $\chi_P(G)$} of $G$ is the minimum $k$ such that $G$ is $k$-paintable.

It follows from the definition that $d$-defective $k$-paintable implies $d$-defective $k$-choosable. The converse is not true. Indeed, although every planar graph is $2$-defective $3$-choosable, it was shown in \cite{GHKZ} that  there are planar graphs that are not $2$-defective $3$-paintable.

On the other hand, it was proved by Schauz \cite{Schauz} that for any graph $G$, $\chi_P(G)\le AT(G)$. So for any graph $G$, 
$ch(G) \le \chi_P(G) \le AT(G)$.  Both  gaps $\chi_P(G)- ch(G)$ and  $AT(G)-  \chi_P(G)$   can   be arbitrarily large \cite{GK}. 
Thus Theorem \ref{thm-GZ} implies that every planar graph is $1$-defective $4$-paintable. We observe that ``having a matching $M$ so that $AT(G-M) \le 4$'' is much stronger than ``
being $1$-defective $4$-paintable''.  
One may compare this to the following results: It is shown in \cite{GHKZ} that every planar graph is $3$-defective $3$-paintable. However, as mentioned earlier,     there are planar graphs $G$ such that for any subgraph $H$ of $G$ with $\Delta(H) \le 3$, $G-E(H)$ is not $3$-choosable \cite{KKZ} (and hence $AT(G-E(H)) \ge 4$).

In this paper, we are interested in the Alon-Tarsi number of some subgraphs of planar graphs without cycles of lengths 4 and $l$ for some $l\in\{5, 6, 7\}$. We denote by $\mathcal{P}_{k, l}$ the family of planar graphs $G$ which contains no cycles of length $k$ or $l$. It was proved in \cite{Lih} that for $l\in\{5, 6, 7\}$, every graph $G\in\mathcal{P}_{4, l}$ is 1-defective 3-choosable. We strengthen this result and prove that for $l\in\{5, 6, 7\}$, every graph $G\in\mathcal{P}_{4, l}$ has a matching $M$ such that $G-M$ has Alon-Tarsi number at most $3$. 
As discussed above, this implies that for $l \in \{5,6,7\}$, every graph $G \in \mathcal{P}_{4,l}$ is $1$-defective $3$-paintable.

\par\bigskip
For a plane graph $G$, we denote its vertex set, edge set and face set by $V(G), E(G)$ and  $F(G)$, respectively.
For a vertex $v$, $d_G(v)$ (or $d(v)$ for short) is the degree of $v$. A vertex $v$ is called a $k$-vertex (respectively,   a $k^+$-vertex  or a $k^-$-vertex) if $d(v)=k$ (respectively, $d(v) \ge k$ or $d(v) \le k$). 
 For $e = uv\in E(G)$, we say $e$ is an $(a, b)$-edge if $d(u)=a$ and $d(v)=b$. For $f\in F(G)$, we denote $f=[u_1u_2\cdots u_n]$ if $u_1, u_2, \cdots, u_n$ are the boundary vertices of $f$ in cyclic order. A 3-face $[u_1u_2u_3]$ is a $(d_1, d_2, d_3)$-face if $d(u_i)=d_i$ for $i=1, 2, 3$.

\section{The main result}

The following is the main result of this paper.

\begin{theorem}
	\label{thm-main}
	For $l \in \{5,6,7\}$, every graph $G \in \mathcal{P}_{4,l}$ has a matching $M$ such that $AT(G-M) \le 3$.
\end{theorem}

For the proof of Theorem \ref{thm-main}, we use an alternate definition of Alon-Tarsi number.
A digraph $D$ is {\em Eulerian} if $d^+_D(v)=d^-_D(v)$ for every vertex $v$. Note that an Eulerian digraph needs not   be connected. In particular, a digraph with no arcs is an Eulerian digraph. Assume $G$ is a graph and $D$ is an orientation of $G$. Let ${\cal E}_e(D)$ (respectively, ${\cal E}_o(D)$) be the set of spanning Eulerian sub-digraphs of $D$ with an even (respectively, an odd) number of arcs.
Let 
$${\rm diff}(D)=  \left| {\cal E}_e(D)\right| -\left| {\cal E}_o(D)\right|.$$
An orientation $D$ of $G$ is {\em Alon-Tarsi (AT)} if ${\rm diff}(D) \ne 0$. Alon and Tarsi \cite{AT} proved that if $D$ is an orientation of $G$, and $\eta(x_v)=d^+_D(v)$, then $  c_{P_{G, \eta}} = \pm {\rm diff}(D)$. Hence the Alon-Tarsi number of $G$ can be defined alternatively as
\begin{center}
	$AT(G)$=min\{$k$: $G$ has an AT orientation $D$ with $\Delta^+_D(v)<k$\}.
\end{center}

The proof of Theorem \ref{thm-main} is by induction. For the purpose of using induction, instead of proving Theorem \ref{thm-main} directly, we shall prove a stronger and more technique result.

\begin{definition} \label{valid-matching} 
	Assume $G$ is a plane graph and $v_0$ is a vertex on the boundary of $G$. A {\em valid} matching of $(G, v_0)$ is a matching $M$ which does not cover $v_0$.
\end{definition}

\begin{definition} \label{good-orientation} 
	 Let $G$ be a plane graph and $v_0$ be a vertex on the boundary of $G$. An orientation $D$ of $G$ is {\em good}, if $D$ is AT with $\Delta_{D}^+(v)<3$ and $d_{D}^+(v_0)=0$.
\end{definition}

We shall prove the following result, which obviously implies Theorem \ref{thm-main}.

\begin{theorem}
	\label{main-thm}
	Assume $l\in\{5, 6, 7\}$, $G\in\mathcal{P}_{4, l}$, and $v_0$ is a vertex on the boundary of $G$. Then $(G, v_0)$ has a valid matching $M$ such that there is a good orientation $D$ of $G-M$.
\end{theorem}

The proof of Theorem \ref{main-thm} uses  discharging method. We shall first describe a family of reducible configurations, i.e., configurations that cannot be contained in a minimum counterexample of Theorem \ref{main-thm}. Then describe a discharging procedure that leads to a contradiction to the Euler's formula. 

We shall frequently use the following lemma in the later proofs.

\begin{lemma}\label{AT}
	 Assume $D$ is a digraph with $V(D)=X_1\cup X_2$ and $X_1 \cap X_2 = \emptyset$. If  all the arcs between $X_1$ and $X_2$ are from $X_1$ to $X_2$. Then $D$ is AT if and only if $D[X_1]$ and $D[X_2]$ are both AT.
\end{lemma}
\begin{proof} 
	Denote by $D_1$ and $D_2$ the sub-digraphs $D[X_1]$ and $D[X_2]$ of $D$, respectively. Note that the set of arcs of an Eulerian digraph can be decomposed into arc disjoint union of directed cycles. Since all the arcs between $X_1$ and $X_2$ are from $X_1$ to $X_2$, and hence none of them is contained in a directed cycle, we conclude that none of these arcs is contained in an Eulerian sub-digraphs of $D$. Hence each Eulerian sub-digraph $H$ of $D$ is the arc disjoint union of an Eulerian sub-digraph $H_1$ of $D_1$ and an Eulerian sub-digraph $H_2$ of $D_2$. Now $H$ is even if and only if $H_1,H_2$ have the same parity. Hence 
	 $|{\cal E}_e(D)|=|{\cal E}_e(D_1)|\times|{\cal E}_e(D_2)|+|{\cal E}_o(D_1)|\times|{\cal E}_o(D_2)|$, $|{\cal E}_o(D)|=|{\cal E}_e(D_1)|\times|{\cal E}_o(D_2)|+|{\cal E}_o(D_1)|\times|{\cal E}_e(D_2)|$. This implies that ${\rm diff}(D) = {\rm diff}(D_1) \times {\rm diff}(D_2)$. Thus, $D$ is AT if and only if $D_1$ and $D_2$ are both AT. \qed
\end{proof}

\par\bigskip
 Assume Theorem \ref{main-thm} is not true and $G$ is a counterexample with   minimum number of vertices. Let $f_0$ denote the outer face of $G$.  
\begin{lemma}\label{2-connected}
	 $G$ is $2$-connected. Moreover, $d_G(v)\ge 3$ for all $v\in V(G)-\{v_0\}$.
\end{lemma}
\begin{proof}
	 Assume $G$ is not $2$-connected. Let $B$ be a block of $G$ that contains a unique cut vertex $z^*$ and does not contain $v_0$. Let $G_1=G-(B-\{z^*\})$. By the minimality, $(G_1, v_0)$ has a valid matching $M_1$ and there is a good orientation $D_1$ of $G_1-M_1$, $(B, z^*)$ has a valid matching $M_2$ and there is a good orientation $D_2$ of $B-M_2$. Let $M=M_1\cup M_2$ and $D=D_1\cup D_2$. Applying Lemma \ref{AT} (with $X_1=V(B)-\{z^*\}$ and $X_2=V(G_1)$), $D$ is AT. So $M$ is a valid matching of $(G, v_0)$, and $G-M$ has a good orientation, a contradiction. 
	 
	 For the moreover part, assume to the contrary that $v\in V(G)-\{v_0\}$ and $d_G(v)\le2$. By induction hypothesis, $G'=G-\{v\}$ has a valid matching $M$ such that $G'-M$ has a good orientation $D'$. Extend $D'$ to an orientation $D$ of $G-M$ in which $v$ is a source vertex. It is obvious that $D$ is a good orientation of $G-M$.  \qed
\end{proof}

\begin{lemma}\label{3adj3} 
	  $G-\{v_0\}$ does not contain two adjacent $3$-vertices.
\end{lemma}
\begin{proof}   Assume to the contrary that $uv\in E(G)$ with $d(u)=d(v)=3$ and $u, v\ne v_0$. Let $G^*=G-\{u, v\}$. Then $(G^*, v_0)$ has a valid matching $M^*$ such that there exists a good orientation $D^*$ of $G^*-M^*$. Let $M=M^*\cup\{uv\}$. Then $M$ is a valid matching of $(G, v_0)$. Extend $D^*$ to an orientation  $D$ of $G-M$ in which $u,v$ are sources.
	Then $D$ is a good orientation of $G-M$.  \qed
\end{proof}

\begin{definition}
	\label{def-minortriangle}
	 A $3$-face $f$  is called a {\em minor triangle} if $f$ is a $(3, 4, 4)$-face and $v_0$ is not on $f$. A $3$-vertex $v$ is called a {\em minor $3$-vertex} if $v$ is incident to a triangle and $v\ne v_0$.
\end{definition}

\begin{definition}
	\label{def-chain}
	A  {\em triangle chain} in $G$ of length $k$ is a subgraph of $G-\{v_0\}$ consisting of vertices $w_1,w_2,\ldots, w_{k+1}, u_1,u_2, \ldots, u_k$ in which $[w_iw_{i+1}u_i]$ is a $(4,4,4)$-face for $i=1,2,\ldots, k$, as depicted in Figure \ref{fig1}(a). 
	We denote by $T_i$ the triangle $[w_iw_{i+1}u_i]$ 
	and denote such a triangle chain by $T_1T_2\ldots T_k$.
	For convenience, a single $4$-vertex is a triangle chain with $0$ triangles.  We say a triangle $T$ intersects a  triangle chain $T_1T_2\dots T_k$, if $T$ has one common vertex with $T_1$.
\end{definition}

\begin{lemma}\label{neighbour-minor} 
  If a minor triangle $T_0$ intersects   a  triangle chain $T_1T_2\dots T_k$, then no vertex of $T_k$ is adjacent to a $3$-vertex, except possibly $v_0$.  In particular, the $k=0$ case implies that no vertex of a minor triangle $T_0$ is adjacent to a $3$-vertex $v \in V(G) - (V(T_0) \cup \{v_0\})$.  
\end{lemma}
\begin{proof}
	Assume to the contrary that  $T_0=[w_0w_1u_0]$ is a minor triangle that intersects   a  triangle chain $T_1T_2\dots T_k$, with  $T_i=[w_iw_{i+1}u_i]$ $(1\le i\le k)$, and $w_{k+1}$ has a neighbour $x$ with  $d(x)=3$, as in Figure \ref{fig1}(b). Assume $w_0$ is a $3$-vertex. Let $X= \cup_{i=0}^kV(T_i) \cup \{x\}$
	and $G'=G-X$. By the minimality of $G$, $(G', v_0)$  contains a valid matching $M'$ and there is a good orientation $D'$ of $G'-M'$.

Let $M=M'\cup   \{w_0u_0, w_1u_1, \dots, w_ku_k, w_{k+1}x\}$.
Then $M$ is a valid matching of $(G, v_0)$. Let $D$ be an orientation of $G-M$ obtained from $D'$ by adding arcs   $(w_i, w_{i+1})$ and $(w_{i+1}, u_i)$ for $i=0, 1, \dots, k$, and all the edges between $X$ and $V-X$ are oriented from $X$ to $V-X$, as depicted in Figure \ref{fig1}(c). Since $D[X]$ is acyclic, $D[X]$ is AT. By Lemma \ref{AT}, $D$ is AT. It is easy to see that $\Delta_{D}^+(v)<3$ and $d^+_D(v_0)=0$. Thus $D$ is a good orientation of $G-M$.
 \qed
\end{proof}

\begin{figure}[htb]
	\centering
	\includegraphics[width=4.6in]{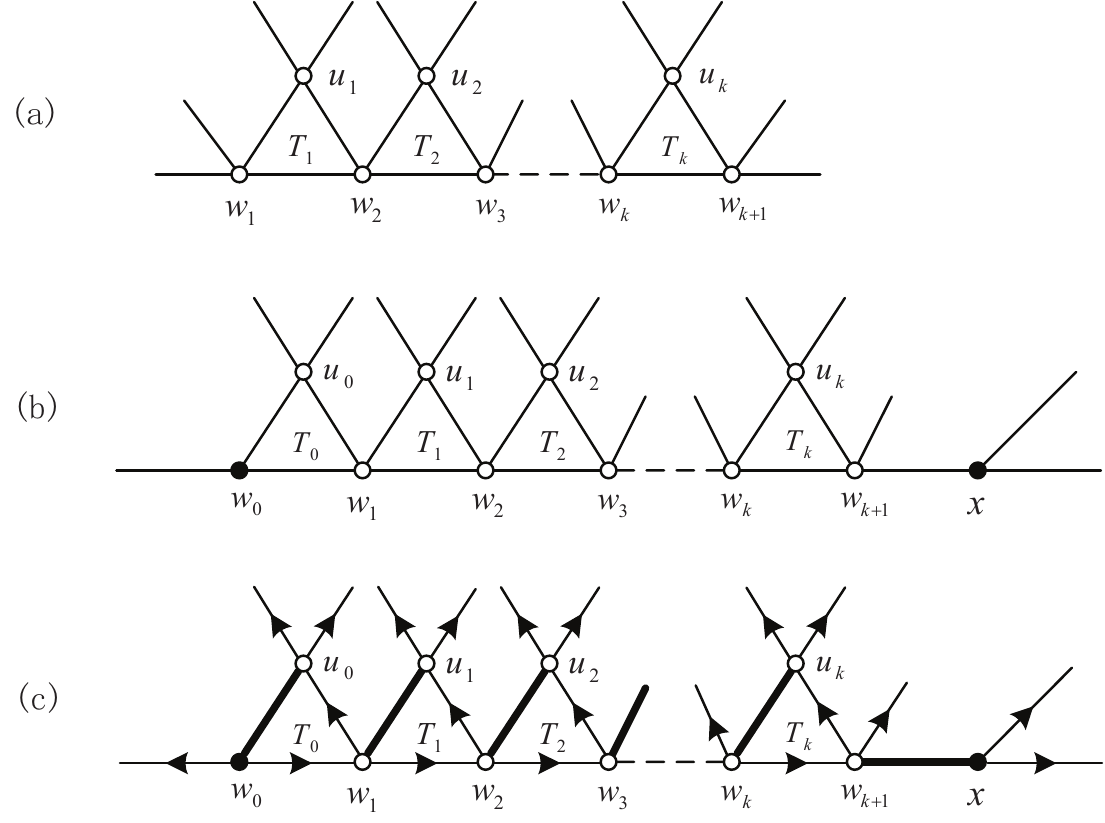}
	\caption{(a) A triangle chain.  (b) The configuration in Lemma \ref{neighbour-minor}. (c) For the proof of Lemma \ref{neighbour-minor}, where a thick line is an edge in the matching $M$.}	
	\label{fig1}
\end{figure} 

\begin{figure}[htb]
	\centering
	\includegraphics[width=4.8in]{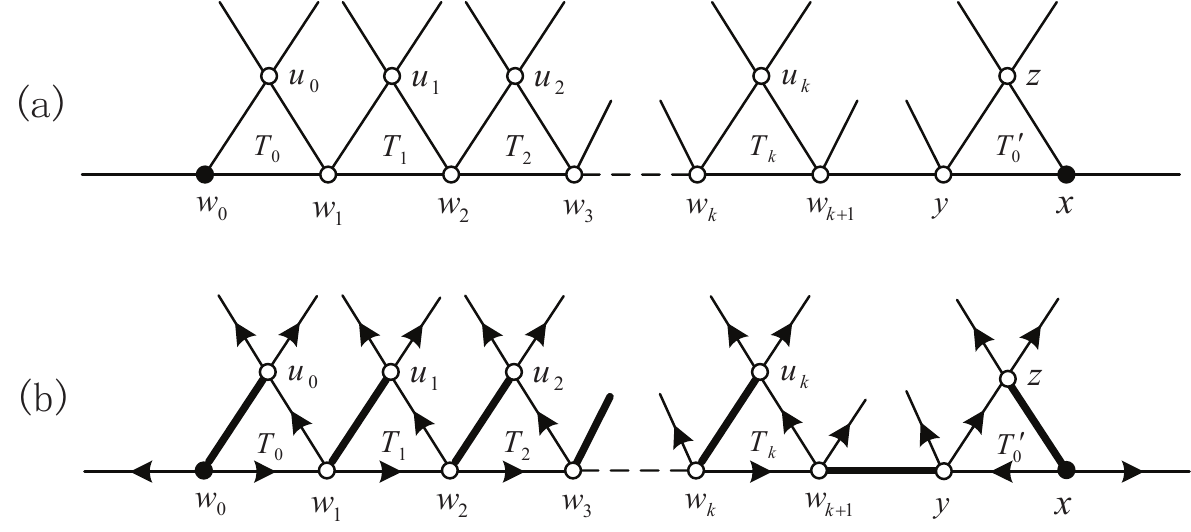}
	\caption{(a) The configuration in Lemma \ref{dist-minor}. (b) For the proof of Lemma \ref{dist-minor}, where a thick line is an edge in the matching $M$.}
	\label{fig11}
\end{figure}
\begin{lemma}\label{dist-minor} 
	   If  a  triangle chain $T_1T_2\dots T_k$ intersects a minor triangle $T_0$, then the distance between $T_k$ and another minor triangle is at least $2$.  In particular, the $k=0$ case implies that any two minor triangles have distance at least $2$. 
\end{lemma}

\begin{proof}
Assume to the contrary that $T_1T_2\dots T_k$ with  $T_i=[w_iw_{i+1}u_i]$ $(1\le i\le k)$ is a triangle chain that intersects a minor triangle $T_0=[w_0w_1u_0]$, and the distance between $T_k$ and another minor triangle  $T_0'=[xyz]$  with $d(x)=3$ is less than 2. By Lemma \ref{neighbour-minor}, we may assume $w_{k+1}y$ is a $(4, 4)$-edge connecting $T_k$ and $T_0'$, as in Figure \ref{fig11}(a). Let $X= \cup_{i=0}^kV(T_i) \cup V(T'_0)$ and $G'=G-X$. Then $(G', v_0)$ has a valid matching $M'$ and there is a good orientation $D'$ of $G'-M'$.

Let $M=M'\cup \{w_0u_0, w_1u_1, \dots, w_ku_k, w_{k+1}y, xz\}$.
Then $M$ is a valid matching of $(G, v_0)$. Let $D$ be an orientation of $G-M$ obtained from $D'$ by adding arcs $(x, y)$,  $(y, z)$, $(w_i, w_{i+1})$ and $(w_{i+1}, u_i)$ for $i=0, 1, \dots, k$, and all the edges between $X$ and $V-X$ are oriented from $X$ to $V-X$, as in Figure \ref{fig11}(b). Obviously, $D[X]$ is acyclic, so $D[X]$ is AT. By Lemma \ref{AT}, $D$ is AT. Additionally, $\Delta_{D}^+(v)<3$ and $d^+_D(v_0)=0$. That is to say, $D$ is a good orientation of $G-M$, a contradiction.\qed
\end{proof}

\par\bigskip

The remainder of the proofs   use a discharging procedure. The \emph{initial charge} $ch$ is defined as: $ch(x)=d(x)-4$ for
$x\in V(G)\cup F(G)$. Applying equalities $\sum\limits_{v\in
	V(G)}d(v)=2|E(G)|=\sum\limits_{f\in F(G)}d(f) $ and Euler's formula
$|V(G)|-|E(G)|+|F(G)|=2$, we conclude that
$$\sum_{x\in V(G)\cup F(G)}ch(x)=-8.$$

In a discharging procedure,
$ch(x\rightarrow y)$ denotes the charge discharged from an element
$x$ to another element $y$,  $ch(x\rightarrow)$ and
$ch(\rightarrow x)$ denote the charge totally discharged from or
to $x$, respectively.
The \emph{final charge} $ch^*(x)$ of $x \in V(G) \cup E(G)$ is defined as $ch^*(x) = ch(x) - ch(x\rightarrow)+ch(\rightarrow x)$.
By applying  appropriate discharging rules, we shall arrive at a final charge  that $ch^*(x)\ge 0$ for all $x\in V(G)\cup F(G)\setminus\{v_0, f_0\}$,  and $ch^*(v_0)+ch^*(f_0) > -8$. As the total charge does not change in the discharging process, this is a contradiction. 

The discharging rules for graphs $G \in \mathcal{P}_{4,l}$ for $l \in \{5,6,7\}$ are different. We use three sections to discuss  graphs $G \in \mathcal{P}_{4,l}$ for $l \in \{5,6,7\}$, respectively.

\section{Planar graphs without $4$- and $5$-cycles}

This section considers plane graphs without $4$- and $5$-cycles. We first derive  more  properties of a minimal counterexample $G$ to Theorem \ref{main-thm}, where $G \in \mathcal{P}_{4,5}$.

\begin{figure}[htb]
	\centering
	\includegraphics[width=5in]{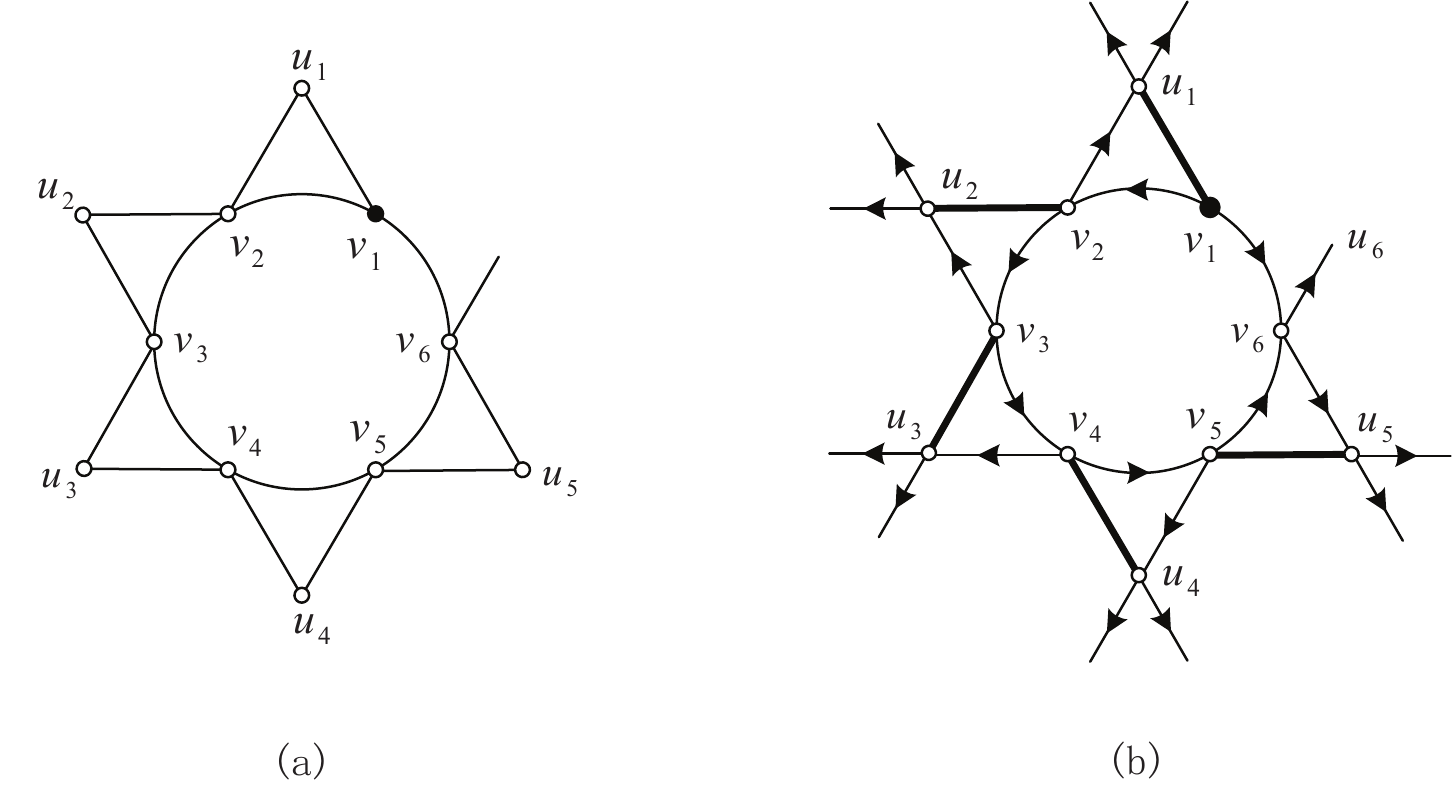}
\caption{(a) The configuration in Lemma \ref{sun}. (b) For the proof of Lemma \ref{sun}, where a thick line is an edge in the matching $M$}
\label{fig3}
\end{figure}

\begin{lemma}\label{sun}
	Assume $f$ is a $6$-face of $G$ which is adjacent to five triangles, and none   of the vertices in these triangles is $v_0$. If $f$ has one $3$-vertex, then there is at least one $5^+$-vertex on the five triangles.
\end{lemma} 
\begin{proof} 
	Let $f=[v_1v_2v_3v_4v_5v_6]$, $v_1$ be a 3-vertex and $T_i=[v_iv_{i+1}u_{i}]$ ($i=1,2,\ldots, 5$) be the five triangles (see Figure \ref{fig3}(a)). 
	Assume  to the contrary that there is no $5^+$-vertex on $T_i$. By Lemma \ref{neighbour-minor}, we may assume all $v_{i+1}$ and $u_i$ are 4-vertices for $i=1,2,\ldots, 5$. Let
	$X= \cup_{i=1}^5V(T_i)$ and $G'= G-X$. Then $(G', v_0)$ has a valid matching $M'$ and there is a good orientation $D'$ of $G'-M'$.

 Let $M = M'\cup\{v_1u_1, v_2u_2,\ldots, v_5u_5\}$. Then $M$ is a valid matching of $(G, v_0)$. Let   $D$ be the orientation of $G-M$   obtained from $D'$ by adding arcs $(v_1, v_6)$ and  $(v_{i+1}, u_{i })$, $(v_i, v_{i+1})$   for $i=1, \dots, 5$, and all the edges between $X$ and $V-X$ are oriented from $X$ to $V-X$ (see Figure \ref{fig3}(b)).   Clearly, $\Delta_{D}^+(v)<3$ and $D$ is AT by Lemma \ref{AT},  a contradiction. \qed
\end{proof}

\par\bigskip

The discharging rules are as follows:
\begin{itemize}
	\item[R1] Assume $f \ne f_0$ is a $3$-face. Then each face adjacent to $f$ transfers   $\frac 13$ charge to $f$.
	\item[R2] Assume $v \ne v_0$ is  $3$-vertex. If $v$ is contained in a triangle, then each of the other two faces incident to $v$ transfers $\frac 12$ charge to $v$; otherwise each face incident to $v$ transfers $\frac 13$ charge to $v$. 
	\item[R3] Assume $u \ne v_0$ is a $5^+$-vertex and $f \ne f_0$ is a $6$-face. If $f$  is adjacent to $s$ triangles that are incident to $u$, then $u$ transfers $\frac{s}{6}$ charge to $f$. 
	\item[R4] $f_0$ transfers $\frac{1}{3}$ charge to each adjacent triangle, and $\frac{1}{2}$  charge to each incident 3-vertex $v \ne v_0$.  $v_0$ transfers $\frac{1}{3}$ charge to each   6-face   $f \ne f_0$ which is either incident to $v_0$, or is not incident to $v_0$ but adjacent to a triangle $T$ which is incident to $v_0$.  
\end{itemize}
  
 
\begin{claim}\label{claim1}
	If a $6$-face $f$ has three minor $3$-vertices, then $ch( \to f) \ge \frac{1}{2}$.
\end{claim} 
\begin{proof} 
	Assume $f=[v_1v_2v_3v_4v_5v_6]$.  By Lemma \ref{3adj3}, we may assume that $v_1$, $v_3$ and $v_5$ are the three minor 3-vertices. Then each of $v_1, v_3, v_5$ is incident to exactly one triangle. Hence at most two  of the three triangles   intersect each other.  Thus we may assume that the three triangles adjacent to $f$ are either $T_1, T_2, T_4$, or $T_1, T_3, T_5$, where $T_i = [v_iv_{i+1}u_i]$. 

\noindent
{\bf Case 1} The three triangles incident to $f$ are  $T_1, T_2, T_4$.

 If $v_0$ is a vertex of $f$ or $T_1,T_2$ or $T_4$, then $v_0$ transfers $\frac{1}{3}$ charge to $f$ by R4. By Lemma \ref{neighbour-minor},   at least  one of the three triangles  have a  $5^+$-vertex $v \ne v_0$ which sends at least $\frac{1}{6}$ charge to $f$. So $ch(\rightarrow f)\ge\frac{1}{3}+\frac{1}{6}=\frac{1}{2}$. 
 
 Assume $v_0$ is not a   vertex of $f, T_1,T_2$ or $T_4$. 
 By Lemma \ref{neighbour-minor}, either $v_2$ is a $5^+$-vertex or both of $u_{1 }$ and $u_{2 }$ are $5^+$-vertices. In both cases, $f$ receives   $\frac{1}{3}$ charge in total from $v_2$, $u_{1 }$ and $u_{2 }$. Moreover, by Lemma \ref{neighbour-minor}, either $v_4$ or $u_{4 }$ is a $5^+$-vertex, which transfers   $\frac{1}{6}$ charge to $f$. Hence,  $ch(\rightarrow f)\ge\frac{1}{3}+\frac{1}{6}=\frac{1}{2}$.

\noindent
{\bf Case 2} The three triangles incident to $f$ are $T_1,T_3,T_5$. 

By Lemma \ref{neighbour-minor},   each of the three triangles has either a $5^+$-vertex or $v_0$ which  transfers at least $\frac{1}{6}$ charge to $f$.   Thus, $ch(\rightarrow f)\ge \frac{1}{2}$. \qed
\end{proof}

\begin{claim}\label{claim2}
If a $6$-face $f$ has two $3$-vertices other than $v_0$ and is adjacent to four triangles, then $ch( \to f) \ge \frac{1}{3}$.
\end{claim}
\begin{proof} Assume $f=[v_1v_2v_3v_4v_5v_6]$ and $T$ is a triangle adjacent to $f$.  If $v_0$ is a vertex of $f$ or $T$, then $v_0$ transfers $\frac{1}{3}$ charge to $f$ by R4. Assume $v_0$ is neither a vertex of $f$ nor a vertex of any triangle  $T$ adjacent to $f$. 
	
	By Lemma \ref{3adj3}, we may assume that either $v_1$ and $v_4$ or $v_1$ and $v_3$ are the two 3-vertices. 
	For $i=1,2,\ldots, 6$, if $v_iv_{i+1}$ is contained in a triangle, then let $T_i=[v_iv_{i+1}u_i]$ be the triangle. We need to consider five cases.

\noindent
{\bf Case 1}  The four triangles incident to $f$ are $T_1, T_2, T_3, T_5$ while the two 3-vertices are $v_1$ and $v_4$. 

If at least one of $v_2$ and $v_3$ is a $5^+$-vertex, then by R3, $ch( \to f) \ge \frac{1}{3}$. Assume both $d(v_2)$ and $d(v_3)$ are 4-vertices. By Lemma \ref{neighbour-minor} and Lemma \ref{dist-minor}, at least two of $u_1$, $u_2$ and $u_3$ are $5^+$-vertices each of which transfers $\frac{1}{6}$ charge to $f$. So $ch(\rightarrow f)\ge\frac{1}{3}$.

\noindent
{\bf Case 2} The four triangles incident to $f$ are $T_1, T_2, T_4, T_5$ while the two 3-vertices are $v_1$ and $v_4$.
 
  By Lemma \ref{neighbour-minor} and R3, at least one of $u_1$, $u_2$, $v_2$ and $v_3$ is a $5^+$-vertex transferring at least $\frac{1}{6}$ charge to $f$. By symmetry, at least one of $u_4$, $u_5$, $v_5$ and $v_6$ transfers at least $\frac16$ charge to $f$. Thus, we are done.

\noindent
{\bf Case 3} The four triangles incident to $f$ are $T_1, T_2, T_4, T_5$ while the two 3-vertices are $v_1$ and $v_3$.

If $v_2$ is a $5^+$-vertex, then $v_2$ transfers $\frac{1}{3}$ charge to $f$ by R3. Assume $v_2$ is a 4-vertex. By Lemma \ref{neighbour-minor}, both $u_1$ and $u_2$ are $5^+$-vertices each of which transfers $\frac{1}{6}$ charge to $f$.

\noindent
{\bf Case 4} The four triangles incident to $f$ are $T_1, T_3, T_4, T_5$ while the two 3-vertices are $v_1$ and $v_3$.

By Lemma \ref{neighbour-minor}, at least one of $v_2$ and $u_1$ is a $5^+$-vertex transferring $\frac{1}{6}$ charge to $f$. Moreover, using Lemma \ref{neighbour-minor} again, at least one of $v_4$, $v_5$, $v_6$, $u_3$, $u_4$ and $u_5$ is a $5^+$-vertex transferring at least $\frac{1}{6}$ to $f$. So $ch(\rightarrow f)\ge\frac{1}{3}$.

\noindent
{\bf Case 5} The four triangles incident to $f$ are $T_3, T_4, T_5, T_6$ while the two 3-vertices are $v_1$ and $v_3$.
 
If one of $v_4$, $v_5$ and $v_6$ is a $5^+$-vertex, then such a $5^+$-vertex sends $\frac{1}{3}$ charge to $f$ by R3. Assume all of $v_4$, $v_5$ and $v_6$ are 4-vertices. Then at least two of $u_3$, $u_4$, $u_5$ and $u_6$ are $5^+$-vertices each sending $\frac{1}{6}$ to $f$. Otherwise, it will contradict to Lemma \ref{neighbour-minor} or Lemma \ref{dist-minor}. Again $ch(\rightarrow f)\ge\frac{1}{3}$.\qed
\end{proof}

\par\bigskip\noindent
$\blacksquare$\quad{\bf Check charge on vertices $v\ne v_0$}

Let $v$ be a 3-vertex. By R2, $v$ gets 1 from incident $6^+$-faces. That is $ch^*(v)=ch(v)-ch(v\to)+ch(\to v)=-1-0+1=0$.

Let $v$ be a 4-vertex. $ch^*(v)=ch(v)=0$.

Let $v$ be a $5^+$-vertex. By R3, $v$ only transfers charge to $6$-faces that are 
adjacent to a triangle incident to $v$. Assume $v$ is incident with $t$ triangles, then $0<t\le\lfloor\frac{d(v)}{2}\rfloor$. Each triangle incident with $v$ is adjacent to at most three $6$-faces, and $v$ transfers $\frac{1}{6}$ to each of the three $6$-faces (note that if a $6$-face $f$ is adjacent to two triangles that are incident to $v$, then $v$ transfers $2 \times \frac 16$ charges to $f$). Hence $v$ sends out at most $\frac{1}{2}t$ charge. So we have $ch^*(v)=ch(v)-ch(v\rightarrow)\ge d(v)-4-\frac{1}{2}t\ge d(v)-4-\frac{1}{2}\times\lfloor\frac{d(v)}{2}\rfloor\ge0$.

\par\bigskip\noindent
$\blacksquare$\quad{\bf Check charge on faces $f\ne  f_0$}

Let $f$ be a 3-face. R1 guarantees $ch^*(f)\ge0$.

Let $f$ be a 6-face. 
 Assume that $f$ has $s$ 3-vertices other than $v_0$. Then $s\le 3$ by Lemma \ref{3adj3}, and $f$ is adjacent to at most $(6-s)$ triangles. 
 
 If $s=0$, then $f$ sends out at most $\frac{1}{3}$ to each adjacent triangle, and hence $ch(f \to) \le \frac 13 \times6=2$ and $ch^*(f) \ge 0$. 
 
 Assume $s=3$. If  $f$ is adjacent to at most two triangles, then $f$ has at most two minor $3$-vertices. So $ch(f\rightarrow)\le\frac{1}{2}\times2+\frac13+\frac{1}{3}\times2=2$  and $ch^*(f) \ge 0$. Assume  $f$ is adjacent to three triangles, then all these three 3-vertices are minor. By Claim \ref{claim1}, we have $ch^*(f)=d(f)-4-ch(f\rightarrow)+ch(\rightarrow f)\ge 2-(\frac{1}{2}\times3+\frac{1}{3}\times3)+\frac{1}{2}=0$.

 Assume  $s=2$. If $f$ is adjacent to at most three triangle, then $ch(f\to)\le \frac12\times2+\frac13\times3=2$ and $ch^*(f) \ge 0$. If $f$ is adjacent to four triangles, then by Claim \ref{claim2},   $ch^*(f)=d(f)-4-ch(f\rightarrow)+ch(\rightarrow f)\ge2-(\frac12\times2+\frac13\times4)+\frac{1}{3}=0$. 
 
Assume $s=1$. If $f$  is adjacent to at most four triangles, then  $ch(f\rightarrow)\le\frac{1}{2}+\frac{1}{3}\times4=\frac{11}{6}<2$.
Assume $f$ is adjacent to five triangles. Then $ch(f\rightarrow)=\frac{1}{2}+\frac{1}{3}\times5=\frac{13}{6}$. On the other hand, either at least one vertex of the five triangles is a $5^+$-vertex  transferring $\frac{1}{6}$ charge to $f$  by Lemma \ref{sun}, or $v_0$ is a vertex of the five triangles transferring $\frac13$ to $f$ by R4. Hence $ch^*(f)\ge 2-\frac{13}{6}+\frac{1}{6}=0$.

Let $f$ be a $7^+$-face. Assume $f$ has $s$ 3-vertices other than $v_0$, then $  s\le\lfloor\frac{d(f)}{2}\rfloor$ and $f$ is adjacent to at most $(d(f)-s)$ triangles. Hence $ch^*(f)=d(f)-4-[\frac{1}{2}\times s+\frac{1}{3}\times(d(f)-s)]=\frac{2}{3}d(f)-\frac{1}{6}s-4 \ge (\frac{2}{3} - \frac 1{12})d(f) -4 >0$.

\par\bigskip\noindent
$\blacksquare$\quad{\bf Check charge on $f_0$ and $v_0$}

By R4, it is clear that $v_0$  transfers at most $(d(v_0)-1)\times\frac{1}{3}$ charge to others. That is, $ch^*(v_0)\ge d(v_0)-4-(d(v_0)-1)\times\frac{1}{3}=\frac{2}{3}d(v_0)-\frac{11}{3}\ge-\frac{7}{3}$ (as $d(v_0)\ge2$).

  Since $f_0$ is incident with at most $\lfloor\frac{d(f_0)}{2}\rfloor$ 3-vertices each getting $\frac{1}{2}$ charge from it, and $f_0$ is adjacent to at most $d(f_0)$ triangles each getting $\frac{1}{3}$ charge from it. We have $ch^*(f_0)\ge d(f_0)-4-\frac{1}{2}\lfloor\frac{d(f_0)}{2}\rfloor-\frac{1}{3}d(f_0)\ge\frac{5}{12}d(f_0)-4\ge-\frac{11}{4}$.

Consequently, we obtain the following contradiction, and the proof is complete.
$$0\le\sum_{x\in V\cup F\setminus\{v_0, f_0\}}ch^*(x)=-8-ch^*(v_0)-ch^*(f_0)\le-\frac{35}{12}.$$

\section{Planar graphs without $4$- and $6$-cycles}

This section shows plane graph without $4$- and $6$-cycles.  We list our discharging rules as follows:

\begin{itemize}
	\item[R1] Assume $f \ne f_0$ is a $3$-face. Then each face adjacent to $f$ transfers   $\frac 13$ charge to $f$.
	\item[R2] Assume $v \ne v_0$ is  $3$-vertex. If $v$ is contained in a triangle, then each of the other two faces incident to $v$ transfers $\frac 12$ charge to $v$; otherwise each face incident to $v$ transfers $\frac 13$ charge to $v$. 
	\item[R3] $f_0$ transfers $\frac{1}{3}$ charge to each adjacent triangle, and $\frac{1}{2}$  charge to each incident 3-vertex $v \ne v_0$.  
\end{itemize}

\par\noindent
$\blacksquare$\quad{\bf Check charge on vertices $v\ne v_0$}

For $d(v)=3$, R2 ensures that the final charge of $v$ is non-negative.

For $d(v)\ge4$, no transference on $v$, we have $ch^*(v)=ch(v)\ge0$.

\par\bigskip\noindent
$\blacksquare$\quad{\bf Check charge on faces $f\ne f_0$}

Let $f$ be a 3-face. R1 guarantees $ch(\to f)=1$. So $ch^*(f)=-1+1=0$.

Let $f$ be a 5-face. Since $G$ does not contain 6-cycle, $f$ is not adjacent to any triangle. Thus $f$ only discharges to the non-minor 3-vertices each of which gets $\frac13$ charge from $f$. On the other hand, $f$ is incident with at most two such 3-vertices by Lemma \ref{3adj3}. It concludes that $ch^*(f)=1-\frac13\times2>0$.

Let $f$ be a $7^+$-face. Assume $f$ is incident with $s$ 3-vertices besides $v_0$. Then by Lemma \ref{3adj3}, $s\le\lfloor\frac{d(f)}{2}\rfloor$. By R1 and R2, $f$ transfers at most $\frac{1}{2}s$ to 3-vertices and $(d(f)-s)\times\frac13$ to triangles. Hence, we have $ch^*(f)\ge d(f)-4-\frac{1}{2}s-\frac{1}{3}(d(f)-s)=\frac{2}{3}d(f)-\frac{1}{6}s-4\ge\frac{7}{12}d(f)-4>0$.

\par\bigskip\noindent
$\blacksquare$\quad{\bf Check charge on $f_0$ and $v_0$}

It is obvious that $ch^*(v_0)=ch(v_0)=d(v_0)-4\ge-2$. 

Since $f_0$ is incident with at most $\lfloor\frac{d(f_0)}{2}\rfloor$ 3-vertices each getting $\frac{1}{2}$ charge from it, and $f_0$ is adjacent to at most $d(f_0)$ triangles each getting $\frac{1}{3}$ charge from it. We have $ch^*(f_0)\ge d(f_0)-4-\frac{1}{2}\lfloor\frac{d(f_0)}{2}\rfloor-\frac{1}{3}d(f_0)\ge\frac{5}{12}d(f_0)-4\ge-\frac{11}{4}$.

Consequently, we obtain the following contradiction, and the proof is complete.
$$0\le\sum_{x\in V\cup F\setminus\{v_0, f_0\}}ch^*(x)=-8-ch^*(f_0)-ch^*(v_0)\le-\frac{13}{4}.$$

\section{Planar graphs without $4$- and $7$-cycles}
In this section, we consider plane graphs without $4$- and $7$-cycles. First we derive  more properties of  a minimal counterexample $G$ to Theorem \ref{main-thm} for $G \in \mathcal{P}_{4,7}$.

\begin{figure}[htb]
	\centering
	\includegraphics[width=5in]{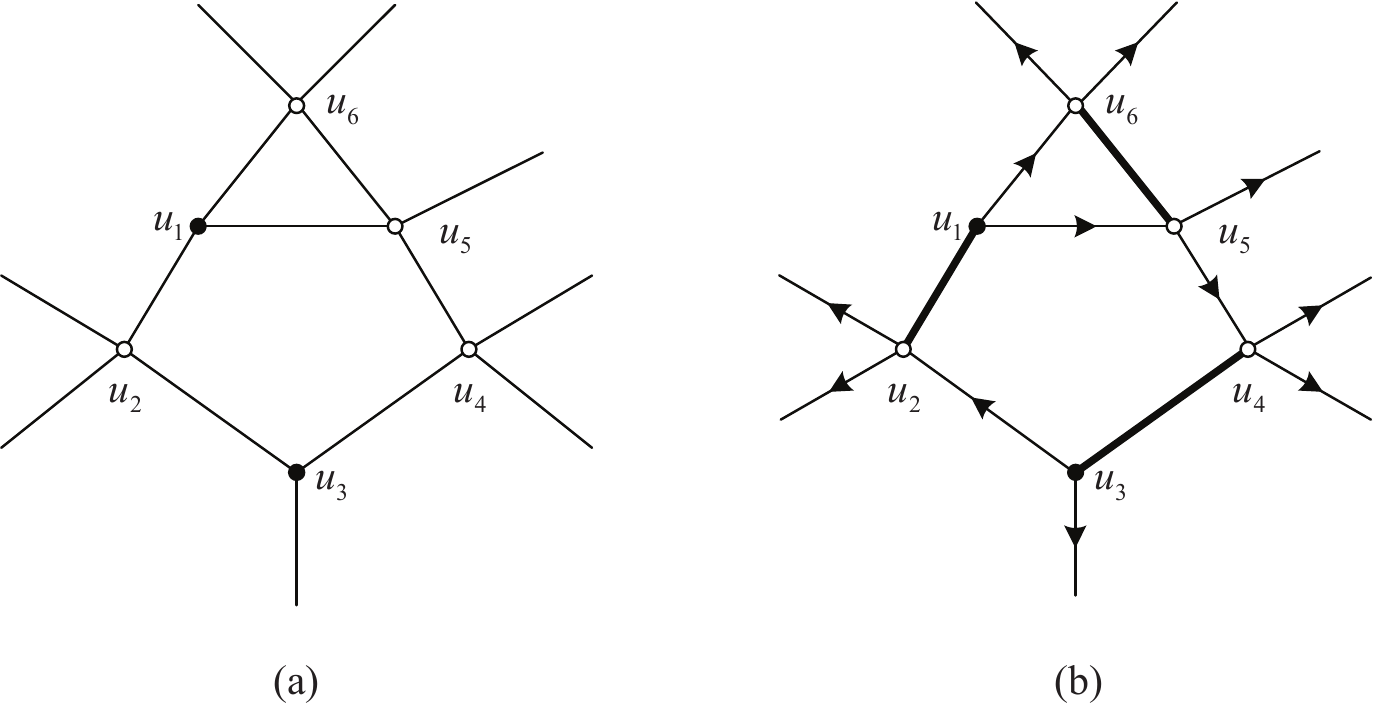}
	\caption{(a) A special $5$-cycle and an adjacent triangle. (b) For the proof of Lemma \ref{STF}, where a thick line is an edge in the matching $M$.}
	\label{fig4}
\end{figure}

\begin{definition}\label{STF}
	 A $5$-cycle   $f=[u_1u_2u_3u_4u_5]$  is called   {\em special  } if   it is adjacent to a triangle $T=[u_1u_5u_6]$ with $u_i\ne v_0$ $(i=1, 2, \dots, 6)$, and all the vertices are $4$-vertices except that $u_1$ and $u_3$ are $3$-vertices, as depicted in Figure \ref{fig4}(a).  
\end{definition}

\begin{lemma}\label{STF}
	$G$ has no special $5$-cycle.
\end{lemma}

\begin{proof} Assume  $f=[u_1u_2u_3u_4u_5]$ is a special $5$-cycle and $T=[u_1u_5u_6]$ is a  triangle adjacent to $f$, where $d(u_1)=d(u_3)=3$ and $d(u_i)=4$ for $i=2, 4, 5, 6$. Let $X=\{u_1, u_2, \ldots, u_6\}$ and $G'=G-X$. Then, by the minimality, $(G', v_0)$ has a valid matching $M'$ and there is a good orientation $D'$ of $G'-M'$. 

Let $M=M'\cup\{u_1u_2, u_3u_4, u_5u_6\}$, then $M$ is a valid matching of $(G, v_0)$. Let $D$ be an orientation of $G$ obtained from $D'$ by adding arcs  $(u_1, u_6)$, $(u_1, u_5)$, $(u_5, u_4)$ and $(u_3, u_2)$, and all the edges between $X$ and $V-X$ are oriented from $X$ to $V-X$, as depicted in Figure \ref{fig4}(b). It is obvious that $D[X]$ is AT. Then, by Lemma \ref{AT}, $D$ is AT. As $\Delta^+_D(v)<3$ and $d^+_D(v_0)=0$, $D$ is a good orientation of $G-M$, a contradiction.\qed
\end{proof}
\par\bigskip
The discharging rules are defined as follows:

\begin{itemize}
	\item[R1] Assume $f \ne f_0$ is a $3$-face. Then each face adjacent to $f$ transfers   $\frac 13$ charge to $f$.
	\item[R2] Assume $v \ne v_0$ is  $3$-vertex. If $v$ is contained in a triangle, then each of the other two faces incident to $v$ transfers $\frac 12$ charge to $v$; otherwise each face incident to $v$ transfers $\frac 13$ charge to $v$. 
	\item[R3] Assume $u \ne v_0$ is a $5^+$-vertex and $f \ne f_0$ is a $5$-face. Then $u$  transfers $\frac{1}{6}$ charge to $f$ either $f$ is  incident to $u$, or $f$ is not incident to $u$ but adjacent to a triangle which is incident to $u$.
	\item[R4] $f_0$ transfers $\frac{1}{3}$ charge to each adjacent triangle, and $\frac{1}{2}$  charge to each incident 3-vertex $v \ne v_0$.  $v_0$ transfers $\frac{1}{3}$ charge to each   5-face   $f \ne f_0$ which is either incident to $v_0$, or is not incident to $v_0$ but adjacent to a triangle $T$ which is incident to $v_0$.  
\end{itemize}

\noindent
$\blacksquare$\quad{\bf Check charge on vertices $v\ne v_0$}

Let $v$ be a 3-vertex. By R2, $ch^*(v)\ge0$.

Let $v$ be a 4-vertex. We have $ch^*(v)=ch(v)=0$.

Let $v$ be a $5^+$-vertex. By R3, $v$ transfers at most $\frac16\times d(v)$ charge to $5$-faces. It follows that $ch^*(v)\ge d(v)-4-\frac{1}{6}d(v)=\frac{5}{6}d(v)-4>0$.

\par\bigskip\noindent
$\blacksquare$\quad{\bf Check charge on faces $f\ne f_0$}

Let $f$ be a 3-face. R1 guarantees $ch^*(f)\ge0$.

Let $f$ be a 5-face. By Lemma \ref{3adj3}, $f$ has at most two 3-vertices other than $v_0$. Since $G$ has no $7$-cycle, $f$ is adjacent to at most one triangle. Namely, $f$ has at most one minor 3-vertex. If $f$ has  at most one 3-vertex other than $v_0$, then $ch(f\rightarrow)\le\frac{1}{2}+\frac{1}{3}<1$. Assume $f$ has two 3-vertices other than $v_0$. 

Firstly, if $f$ does not have any minor 3-vertex, then $f$ transfers at most $\frac{1}{3}$ charge to the unique triangle and $\frac{1}{3}\times2$ to the non-minor 3-vertices. That is, $ch^*(f)=ch(f)-ch(f\to)\ge1-1=0$. 

Assume $f$ has a minor 3-vertex. Assume  $f=[v_1v_2v_3v_4v_5]$ with $d(v_1)=3$ and  $T=[v_1v_2v_6]$. In this case, $ch(f\rightarrow)=\frac{1}{2}+\frac{1}{3}+\frac{1}{3}=\frac76$. If one of $v_i$ ($i=2, 3, \dots, 6$) is $v_0$ or a $5^+$-vertex, then such $v_i$ transfers at least $\frac16$ charge to $f$ by R4 and R3. Thus, $ch^*(f)=ch(f)-ch(f\to)+ch(\to f)\ge1-\frac76+\frac16\ge0$. Assume  $f$ and $T$ does not contain $v_0$ and $5^+$-vertex. By Lemma \ref{3adj3} and Lemma \ref{neighbour-minor}, another 3-vertex must be $v_4$. Thus, there is a special $5$-cycle in $G$, contradicting to Lemma \ref{STF}.

Let $f$ be a 6-face. By Lemma \ref{3adj3}, $f$ has at most three 3-vertice other than $v_0$. Since $G$ has no $4$- and $7$-cycles, $f$ is adjacent to at most one triangle $T$ which shares two common edges with $f$. If $f$ is not adjacent to any triangle, $f$ only sends charge to non-minor 3-vertices each getting $\frac13$ from $f$. Hence, $ch^*(f)\ge d(f)-4-\frac{1}{3}\times3>0$. If $f$ is adjacent to one triangle $T$ which shares two common edges with $f$, then $f$ has at most one minor 3-vertex. Thus, $ch^*(f)\ge d(f)-4-\frac12-\frac{1}{3}\times2-\frac13>0$.

Let $f$ be a $8^+$-face. If $f$ is incident with $s$ 3-vertices other than $v_0$ where $0\le s\le\lfloor\frac{d(f)}{2}\rfloor$. Then $f$ transfers at most $\frac12\times s$ charge to 3-vertices and $(d(f)-s)\times\frac13$ charge to triangles. Hence, we have $ch^*(f)\ge d(f)-4-\frac{1}{2}s-\frac{1}{3}(d(f)-s)=\frac{2}{3}d(f)-\frac{1}{6}s-4\ge\frac{7}{12}d(f)-4>0$.

\par\bigskip\noindent
$\blacksquare$\quad{\bf Check charge on $f_0$ and $v_0$}

For this checking procedure is the same as the last part in Section 3, we omit the details. That is $ch^*(v_0)\ge-\frac{7}{3}$ and $ch^*(f_0)\ge-\frac{11}{4}$.

Thus, we will have $0\le\sum_{x\in V\cup F\setminus\{v_0, f_0\}}ch^*(x)=-8-ch^*(v_0)-ch^*(f_0)\le-\frac{35}{12}$, a contradiction.

\end{document}